\documentclass[a4paper,10pt,reqno]{article}
\usepackage[utf8]{inputenc}
\usepackage[T1]{fontenc}
\usepackage[english]{babel}
\usepackage[a4paper,left=2.5cm, right=2.5cm, top=2.7cm, bottom=2.7cm,%
includehead,includefoot]{geometry}

\usepackage{amsmath,amssymb,amsfonts,amsthm,amscd}
\usepackage[all,cmtip]{xy}
\usepackage{graphicx}
\usepackage{subfigure}
\usepackage{float}
\usepackage{xcolor}
\usepackage[numbers,square,sort&compress]{natbib}
\usepackage{cases}
\usepackage{booktabs}
\usepackage{colortbl}
\usepackage{caption}
\usepackage{enumitem}
\usepackage{epstopdf}
\usepackage{algorithm}
\usepackage{algpseudocode}
\usepackage{array}
\usepackage{bbding}
\usepackage{fancyhdr} 
\usepackage{fancyvrb} 
\usepackage{longtable}
\usepackage{listings}
\usepackage{rotating,rotfloat} 
\usepackage{yhmath} 
\usepackage{latexsym}
\usepackage{soul}
\usepackage{tikz}
\usepackage{hyperref}
\hypersetup{colorlinks,citecolor=black,linkcolor=black,urlcolor=black,breaklinks=true}
\usepackage{url}
\usepackage{appendix}
\usepackage{titlesec}
\usepackage{todonotes}
\setuptodonotes{size=\tiny}
\usepackage{fancyhdr}
\numberwithin{equation}{section}
\newtheorem{theoremcounter}{theoremcounter}[section]

\allowdisplaybreaks

\theoremstyle{plain}

\newtheorem{corollary}[theoremcounter]{Corollary}
\newtheorem{lemma}[theoremcounter]{Lemma}
\newtheorem{proposition}[theoremcounter]{Proposition}

\newtheorem{theorem}[theoremcounter]{Theorem}

\theoremstyle{definition}
\newtheorem{definition}[theoremcounter]{Definition}
\newtheorem{remark}[theoremcounter]{Remark}

\usepackage[nameinlink, capitalize, noabbrev]{cleveref}

\newcommand{\bZ}{\mathbb{Z}}

\newcommand{\bbm}{\begin{bmatrix}}
	\newcommand{\ebm}{\end{bmatrix}}

\newcommand{\authors}{Baiying Ren}
\renewcommand{\title}{Higher Kazhdan projections and delocalized $\ell^2$-Betti numbers for an amalgamated product group}

\begin{document}
\titleformat{\section}
  {\normalfont\Large\bfseries}{\thesection}{0.5em}{}
\titleformat{\subsection}
  {\normalfont\large\bfseries}{\thesubsection}{0.4em}{}

    \pagestyle{plain}
	
	\noindent
	\begin{minipage}{\linewidth}
		\begin{center}
			\textbf{\Large \title} \\
			\authors    
		\end{center}
	\end{minipage}
	
	\renewcommand{\thefootnote}{}
	\footnotetext{%
	}
	
	\vspace{2em}
	\noindent
	\begin{minipage}{\linewidth}
		\textbf{Abstract}. 
        We establish explicit expressions for the $K$-theory classes of higher Kazhdan projections for amalgamated product groups $\bZ_m*_{\bZ_d}\bZ_n$.
        Our approach follows the methodology developed by Pooya and Wang for free product groups $\bZ_m*\bZ_n$, and naturally generalizes their results on free products.
        As an application of the $K$-class expressions, we obtain non-vanishing results for delocalized $\ell^2$-Betti numbers of $\mathrm{SL}(2,\bZ)$. 
	\end{minipage}

\section{Introduction}
\label{section1}
In operator algebras and group theory, Kazhdan projections are specific projections constructed for a group $G$ with Kazhdan's property $(T)$.
They lie in the maximal group $C^*$-algebra $\mathrm{C^*_{max}}(G)$ and are well-known to play a role in constructing counterexamples to the coarse Baum-Connes conjecture \cite{Higson} and to the Baum-Connes conjecture with coefficients \cite[Section 5]{HigsonLafforgueSkandalis}. 
In \cite{LNP}, Li, Nowak, and Pooya generalized the Kazhdan projections to higher degrees, establishing a relation between $\ell^2$-Betti numbers for groups and the surjectivity of Baum–Connes type assembly maps.
To be more precise, the Kazhdan projection for a discrete group $G$ in degree $n$, associated with a family of unitary representations $\varPsi$, is a certain projection $p_n$ in the matrix algebra over $\mathrm{C^*_{\varPsi}}(G)$. 
The Kazhdan projection $p_n$ is required to satisfy that for each representation $(\pi,H)\in \varPsi$, $\pi(p_n)$ is the orthogonal projection onto the space of harmonic $n$-cochains with coefficients in $(\pi,H)$.
Especially, when $G$ is a property $(T)$ group and $\varPsi$ is the family of all its unitary representations, the associated projection $p_0\in \mathrm{C^*_{max}}(G)$ exists and coincides with the classical Kazhdan projection.

When $\varPsi$ is the left regular representation of $G$, the higher Kazhdan projections $p_n$ can give rise to non-zero $K$-classes of the reduced group $C^*$-algebra $\mathrm{C^*_{r}}(G)$, that is, 
$$[p_n]\in \mathrm{K_0(C^*_r}(G)), \quad n=0,1,\cdots.$$
Notably, these $K$-classes provide a connection between higher Kazhdan projections and fundamental topological invariants, specifically $\ell^2$-Betti numbers and their delocalized counterparts, via trace maps.
On the one hand, by \cite[Proposition 7]{LNP}, the pairing of $[p_n]$ with the canonical von Neumann trace $\tau$ on $\mathrm{C^*_r}(G)$ yields the $n$-th $\ell^2$-Betti number $\beta ^{(2)}_n(G)$. (We refer to \cite{Lueck} for a comprehensive introduction to $\ell^2$-Betti numbers.)
On the other hand, Lott \cite{Lott} introduced delocalized $\ell^2$-Betti numbers for discrete groups $G$, extending his earlier definition for closed Riemannian manifolds.
When $G$ is a hyperbolic group, Puschnigg \cite{Puschnigg} further extended the delocalized traces $\tau_{\langle g\rangle}$ from $\ell^1(G)$ to a specific smooth subalgebra of $\mathrm{C^*_r}(G)$, which ensures that the pairing between $\mathrm{K_0(C^*_r}(G))$ and $\tau_{\langle g\rangle}$ is well-defined.
Building on this framework, Pooya and Wang \cite{Pooya-Wang} demonstrated that for a hyperbolic group $G$, if $p_n$ exists in the matrix algebra over $\mathrm{C^*_r}(G)$, then the pairing of $[p_n]\in \mathrm{K_0(C^*_r}(G))$ with the delocalized trace $\tau_{\langle g\rangle}$ recovers the delocalized $\ell^2$-Betti number $\beta ^{(2)}_{n, \langle{g}\rangle}(G)$. 
This allowed them to propose a $K$-theoretic definition of delocalized $\ell^2$-Betti numbers for a specific class of groups (see \cite[Definition 2.11]{Pooya-Wang}) and establish non-vanishing calculations through this formulation.


However, explicit calculation of $K$-theory classes $[p_n]\in \mathrm{K_0(C^*_r}(G))$ is quite challenging.
Nevertheless, Pooya and Wang \cite{Pooya-Wang} established an expression for the class $[p_1]$ of the first higher Kazhdan projections for free product groups $G=\bZ_m*\bZ_n$ by direct computation. (Their results on Cartesian product groups are beyond the scope of discussion in this paper.)
Due to Lemma 4.1 in \cite{Pooya-Wang}, $p_1$ is also the only non-vanishing Kazhdan projection for 
$G$ since $G$ is virtually free.
In the following, we recall their $K$-class expressions for free products and their calculations of delocalized $\ell^2$-Betti numbers, particularly for $G=\mathrm{PSL}(2,\bZ)=\bZ_2*\bZ_3$.
\begin{theorem}[\cite{Pooya-Wang}]
\label{PW}
    The $K$-theory class of the first higher Kazhdan projection for $$G=\bZ_m*\bZ_n=\langle s,t|s^m=1,t^n=1 \rangle$$ with $m\geq 2$ and $n\geq 3$, can be represented by 
  	\begin{equation*}
  		[p_1]=[1]-[\frac{\sum_{i=0}^{m-1}s^i}{m}]-[\frac{\sum_{i=0}^{n-1}t^i}{n}]\in \mathrm{K_0(C^*_r}(G)).
  	\end{equation*}
\end{theorem}
Throughout this paper, given an element $g\in G$, we denote by $\langle{g}\rangle$ its conjugacy class.
\begin{theorem}[\cite{Pooya-Wang}] The delocalized $\ell^2$-Betti numbers for
$$G=\mathrm{PSL}(2,\bZ)=\bZ_2*\bZ_3=\langle s,t|s^2=1,t^3=1 \rangle$$ 
are given by 
\begin{equation*}
			\beta^{(2)}_{1, \langle{g}\rangle}(G) = 
			\begin{cases}
				1/6 &\qquad  g=1\\
                -1/2  &\qquad  g\in\langle{s}\rangle\\
				-1/3  &\qquad  g\in\langle{t}\rangle\ \mathrm{or} \ \langle{t^2}\rangle\\ 
				0    &\qquad \text{otherwise} 
			\end{cases}
		\end{equation*}
		and $\beta ^{(2)}_{k, \langle{g}\rangle}(G) = 0$ for $k \neq 1$ and $g\in G$.
    
\end{theorem}
In this paper, we aim to extend the above $K$-class expressions  $[p_1]\in \mathrm{K_0(C^*_r}(G))$ to specific amalgamated product groups and use these to obtain more examples of delocalized $\ell^2$-Betti numbers.



Our main result is the generalization of Theorem \ref{PW} to amalgamated product groups $G=\bZ_m*_{\bZ_d}\bZ_n$, presented as follows.
It is easy to see that Theorem \ref{PW} is a special case of Theorem \ref{main th} when $d=1$.
\begin{theorem}
\label{main th} 
   The $K$-theory class of the first higher Kazhdan projection for $$G=\bZ_m*_{\bZ_d}\bZ_n=\langle s,t|s^m=1,t^n=1,r:=s^{\frac{m}{d}}=t^{\frac{n}{d}} \rangle$$ 
   with $d|m$, $d|n$, $d\neq m$, $d\neq n$, $m \geq 2$ and $n\geq 3$, can be represented by
  	\begin{equation*}
  		[p_1]=[\frac{\sum_{i=0}^{d-1}r^i}{d}]-[\frac{\sum_{i=0}^{m-1}s^i}{m}]-[\frac{\sum_{i=0}^{n-1}t^i}{n}]\in \mathrm{K_0(C^*_r}(G)).
  	\end{equation*}
  \end{theorem}
The more general conclusion in this paper not only provides a deeper understanding of higher Kazhdan projections, but also yields more examples of non-vanishing delocalized $\ell^2$-Betti numbers.
Pairing both sides of the formula for $[p_1]$ in Theorem \ref{main th} with delocalized traces $\tau_{\langle g\rangle}$, we can obtain delocalized $\ell^2$-Betti numbers for $G=\bZ_m*_{\bZ_d}\bZ_n$. In particular, we present the results for $G=\mathrm{SL}(2,\bZ)=\bZ_4*_{\bZ_2}\bZ_6$ as follows.
\begin{corollary}
\label{corollary:SL2}
   The delocalized $\ell^2$-Betti numbers for 
   $$G=\mathrm{SL}(2,\bZ)=\bZ_4*_{\bZ_2}\bZ_6=\langle s,t|s^4=1,t^6=1,r:=s^2=t^3 \rangle$$  
   are given by
    \begin{equation*}
			\beta^{(2)}_{1, \langle{g}\rangle}(G) = 
			\begin{cases}
                1/12  &\qquad  g=1 \ \mathrm{or}\ g\in\langle{r}\rangle= \langle{s^2}\rangle = \langle{t^3}\rangle\\
				-1/4  &\qquad  g\in\langle{s}\rangle\ \mathrm{or} \ \langle{s^3}\rangle\\
				-1/6  &\qquad  g\in\langle{t}\rangle,\langle{t^2}\rangle,\langle{t^4}\rangle\ \mathrm{or}\ \langle{t^5}\rangle \\
				0    &\qquad \text{otherwise} 
			\end{cases}
		\end{equation*}
		and $\beta ^{(2)}_{k, \langle{g}\rangle}(G) = 0$ for $k \neq 1$ and $g\in G$.
\end{corollary}

The generalization in this paper is sophisticated.
We explain the main challenges as follows. 
First, to calculate the first higher Kazhdan projection for $G=\bZ_m*_{\bZ_d}\bZ_n$, one needs to consider a model $X$ of $BG$ with a finite 2-skeleton and express the first Laplacian $\Delta_1$ associated with the cellular chain complex of $X$ with coefficients in $\ell^2(G)$. 
To address this, we employ Gruenberg's free resolution of the trivial $\bZ G$-module $\bZ$ \cite{Gruenberg1960}.
Second, computing harmonic forms $\mathrm{ker}\Delta_1$ in this broader case is more intricate and technical.
In fact, this constitutes the core contribution of our paper, and we dedicate Section 3 to presenting the detailed computational arguments.
Finally, to express the $K$-class of the first higher Kazhdan projection $p_1$ (which is precisely the orthogonal projection onto $\mathrm{ker}\Delta_1$) in $\mathrm{K_0(C^*_r}(G))$, we employ orthogonal direct sum decompositions within copies of $\ell^2(G)$. 
This approach originally follows the techniques of Pooya and Wang \cite{Pooya-Wang}, which can be naturally generalized to our broader setting.

The remainder of this paper is organized as follows. 
In Section \ref{section2}, we recall the necessary preliminaries including higher Kazhdan projections, delocalized $\ell^2$-Betti numbers, and a free resolution of finitely presented groups.
In Section \ref{section3}, we calculate the kernel of the first Laplacian $\Delta_1$ for $G=\bZ_m*_{\bZ_d}\bZ_n$.
In Section \ref{section4}, we present the proof of our main theorem, Theorem \ref{main th}, concerning the $K$-theory class of the first higher Kazhdan projection for $G=\bZ_m*_{\bZ_d}\bZ_n$.

\section*{Acknowledgements} 
I would like to express my profound gratitude to my supervisor Hang Wang for suggesting the topic of this paper and for her invaluable guidance and continuous support throughout the research. 
I am also sincerely thankful to Sanaz Pooya for numerous inspiring discussions that greatly improved this work.

\section{Preliminaries}
\label{section2}
\subsection{Higher Kazhdan projections}


In this subsection, we briefly recall the definition of higher Kazhdan projections in \cite{LNP}.
Let $G$ be a group of type $F_{n+1}$ with a chosen CW-complex model $X$ of $BG$ with finite $n+1$-skeleton.
Denote by $\bZ G$ the group ring of $G$ over $\bZ$.
Then the cellular chain complex of $X$ implies a free resolution of the trivial $\bZ G$-module $\bZ$:
\begin{equation}
\tag{*} \label{free resolution}
   \cdots \to (\bZ G)^{\oplus k_{n+1}} \stackrel{\delta_{n}}{\longrightarrow}  (\bZ G)^{\oplus k_n} \stackrel{\delta_{n-1}}{\longrightarrow} \cdots  \longrightarrow (\bZ G)^{\oplus k_{1}} \stackrel{\delta_{0}}{\longrightarrow}   \bZ G  \stackrel{\epsilon}{\longrightarrow} \bZ \longrightarrow 0,
   \end{equation}
where for $l\leq n+1$, $k_l$ denotes the number of $l$-cells of $X$. 
Denote by $\mathrm{M}_{m\times m^{\prime}}(\bZ G)$ the algebra of $m\times m^{\prime}$ matrices over $\bZ G$, and by $\mathrm{M}_{m}(\bZ G)$ when $m=m^{\prime}$.
In the above resolution (\ref{free resolution}), $\delta_{i}$ denotes the associated boundary differential that belongs to $\mathrm{M}_{k_{i}\times k_{i+1}}(\bZ G)$ for $0 \leq i \leq
 n$, and $\epsilon$ denotes the augmentation map associated with $G$ given by
   \begin{align*}
   	\epsilon: \bZ G &\to \bZ\\
   	\sum_{s} a_sg_s &\mapsto \sum_{s} a_s, \ \mathrm{where\ the\ sum\ is\ finite \ and}\ a_s\in \bZ,\ g_s\in G.
   \end{align*}
Clearly, it is a homomorphism of $\bZ G$-modules.
   
Let $(\pi,H)$ be a unitary representation of $G$. Then after applying the functor $\mathrm{Hom}_{\bZ G}(\cdot,H)$ to the free resolution (\ref{free resolution}) and identifying the terms with direct sums of $H$, we obtain the induced cochain complex with coefficients in $(\pi,H)$:
    \begin{equation*}
    	H \stackrel{\pi(d_0)}{\longrightarrow} H^{\oplus k_1} \stackrel{\pi(d_1)}{\longrightarrow} H^{\oplus k_2} \longrightarrow \cdots.
    \end{equation*}
The above codifferentials are represented by elements
$$
d_i \in \mathrm{M}_{k_{i+1}\times k_{i}}(\bZ G),\quad 0 \leq i \leq
 n,
$$
which induce the following operators under the unitary representation $(\pi,H)$:
$$
\pi(d_i): H^{\oplus k_{i}} \to H^{\oplus k_{i+1}},\quad 0 \leq i \leq n.
$$
The Laplacian in degree $n$ is then defined by the element
$$
\Delta_n=d_n^*d_n+d_{n-1}d_{n-1}^* \in \mathrm{M}_{k_{n}}(\bZ G).
$$
The kernel of $\pi(\Delta_n)$ on $H^{\oplus k_n}$ is the space of harmonic $n$-cochains with coefficients in $(\pi,H)$ and the Hodge-de Rham isomorphism implies its connection with the reduced group cohomology:
$$
\mathrm{ker}\,\pi(\Delta_n) \cong
\overline{H}^n(G,\pi).
$$
\begin{definition}[\cite{LNP}, Definition 4]
    Let $G$ be a discrete group of type $F_{n+1}$ with a family of unitary representations $\varPsi$.
    The Kazhdan projection in degree $n$ is a specific projection $p_n\in \mathrm{M}_{k_{n}}(\mathrm{C^*_{\varPsi}}(G))$ such that for every unitary representation $(\pi,H)\in \varPsi$, $\pi(p_n)$ is the orthogonal projection onto $\mathrm{ker}\,\pi(\Delta_n)$. 
\end{definition}

 In the following, we present propositions characterizing the existence of Kazhdan projections (established in \cite{LNP}) and laying a foundation for the connection between their existence and cohomology groups (established in \cite{BaderNowak}).
 In particular, for virtually free groups, non-vanishing Kazhdan projections exist only in degree one, according to \cite{Pooya-Wang}.
 
\begin{proposition}[\cite{LNP}, Proposition 5]
    If $\Delta_n\in \mathrm{M}_{k_{n}}(\bZ G) \subseteq \mathrm{M}_{k_{n}}(\mathrm{C^*_{\varPsi}}(G))$ has a spectral gap at 0, then the Kazhdan projection $p_n$ exists in the $C^*$-algebra $\mathrm{M}_{k_{n}}(\mathrm{C^*_{\varPsi}}(G))$.
\end{proposition}

\begin{proposition}[\cite{BaderNowak}, Proposition 16]
    Let $\pi$ be a unitary representation of $G$.
    The Laplacian $\pi(\Delta_n)$ has a spectral gap at 0 if and only if the cohomology groups $H^n(G,\pi)$ and $H^{n+1}(G,\pi)$ are both reduced. 
\end{proposition}

\begin{lemma}[\cite{Pooya-Wang}, Lemma 4.1]
\label{lemma: only p1 nonvainishing}
    Let $G$ be a virtually free group. Then $p_1$ lies in matrices over $\mathrm{C^*_{r}}(G)$ and defines a non-zero class $[p_1]\in \mathrm{K_0(C^*_{r}}(G))$.
    All other $p_n$ vanish when $n\neq 1$.
\end{lemma}

In this paper, we focus on the case where $\varPsi$ is the left regular representation $(\lambda,\, \ell^2(G))$ of $G$. Therefore, $\mathrm{C^*_{\varPsi}}(G)$ coincides with the reduced group $C^*$-algebra $\mathrm{C^*_{r}}(G)$.
For simplicity, we denote the space of harmonic $n$-cochains with coefficients in $(\lambda,\, \ell^2(G))$ by $\mathrm{ker}\Delta_n$ instead of $\mathrm{ker}\,\lambda(\Delta_n)$ throughout this paper.

\subsection{Delocalized $\ell^2$-Betti numbers}

In this subsection, we briefly recall the $K$-theoretic definition of delocalized $\ell^2$-Betti numbers in \cite{Pooya-Wang}.
Let $G$ be a discrete group and $g\in G$. The delocalized trace $\tau_{\langle g \rangle}$ on $\ell^1(G)$ is the following bounded linear map with trace property
\begin{align*}
\tau_{\langle g \rangle}: \ell^1(G) &\to \mathbb{C},\\
f &\mapsto \tau_{\langle g \rangle}(f)=\sum_{h\in \langle g \rangle}f(h).
\end{align*}
Recall that $\langle g \rangle$ denotes the conjugacy class of $g$.

\begin{definition}[\cite{Pooya-Wang}, Definition 2.11]
    Let $G$ be a discrete group of type $F_{n+1}$ and $g\in G$.
    Assume that there exists a smooth subalgebra $\mathcal{S}\subseteq \mathrm{C^*_r}(G)$--that is, a dense Banach subalgebra and closed under functional calculus--such that the delocalized trace $\tau_{\langle g \rangle}$ admits a continuous extension to it.
    Assuming further the existence of the higher Kazhdan projection $p_n\in \mathrm{M}_{k}(\mathrm{C^*_{r}}(G))$ for some $k$, the $n$-th delocalized $\ell^2$-Betti number of $G$ is defined via the pairing of its $K$-theory class $[p_n]\in \mathrm{K_0(C^*_r}(G))$ with $\tau_{\langle g \rangle}$:
    $$
    \beta ^{(2)}_{n, \langle{g}\rangle}(G)=\tau_{\langle g \rangle}([p_n]).
    $$
\end{definition}

\subsection{A free resolution of finitely presented groups}
\label{subsection: Gruenberg resolution}
For a finitely presented group $G$, Gruenberg \cite{Gruenberg1960,Gruenberg1970} constructed a free resolution associated with $G$ by studying its relation module and building upon Fox's theory of derivatives in group rings \cite{Fox}. 
The Gruenberg resolution generalizes Lyndon's earlier work on one-relator groups (see \cite{Lyndon} for details).
In the following theorem, we present the Gruenberg resolution, but restrict ourselves to the first two differentials.
Geometrically, these differentials arise from a particular CW-complex model of $BG$ whose 2-skeleton is finite and is given by the presentation complex associated with $G$.

\begin{theorem}[\cite{Gruenberg1960}, Theorem 1]
   	\label{th1} 
   Let $G$ be a finitely presented group given by a presentation
   \begin{equation*}
   	G=\langle f_1,\cdots,f_d| r_1,\cdots,r_t \rangle.
   \end{equation*}
   Regard this presentation as a short exact sequence of groups
   \begin{equation*}
   	 1\to R\to F \to G \to 1,
   \end{equation*}
   in which
   \begin{equation*}
   	F=\langle f_1,\cdots,f_d \rangle,\ R=\langle r_1,\cdots,r_t \rangle
   \end{equation*}
   are free.
   Then there exists a free resolution of the trivial $\bZ G$-module $\bZ$:
   \begin{equation}
   	\label{resolution}
   	\cdots \to (\bZ G)^{\oplus t} \stackrel{\delta_1}{\longrightarrow}  (\bZ G)^{\oplus d} \stackrel{\delta_0}{\longrightarrow}\bZ G \stackrel{\epsilon}{\longrightarrow} \bZ \longrightarrow 0,
   \end{equation}
   with
   \begin{equation*}
       \delta_0=\left[\begin{matrix}
           1-f_1 & \cdots & 1-f_d
       \end{matrix}\right],\ \ \delta_1=\left[\begin{matrix}
           \dfrac{\partial r_j}{\partial f_i}
       \end{matrix}\right]_{1\leq i\leq d,1\leq j \leq t}.
   \end{equation*}
   Here the Fox derivatives $\frac{\partial r_j}{\partial f_i}\in \bZ G$ are defined by first specifying elements in $\bZ F$ using the following formulas for $\bZ F$, and then projecting them to $\bZ G$ via the map $\bZ F \to \bZ G$:
   \begin{align*}
       \frac{\partial 1_F}{\partial f_i}&=0,\\
       \frac{\partial f_i}{\partial f_k}&=\begin{cases}
           1,\ \mathrm{if}\  i=k\\
           0,\ \mathrm{if}\ i\neq k
       \end{cases},\\
       \frac{\partial uv}{\partial f_i}&=\frac{\partial u}{\partial f_i}+u\frac{\partial v}{\partial f_i},
   \end{align*}
   where $1_F$ represents the group identity of $F$, $u,v\in F$, and $1\leq i,k\leq d$. 
\end{theorem}


\section{The $\ell^2(G)$-kernel of the Laplacian for $G=\bZ_m*_{\bZ_d}\bZ_n$}
\label{section3}
In this section, we compute the kernel of the first Laplacian $\mathrm{ker}\Delta_1$ with coefficients in $\ell^2(G)$ for $G=\bZ_m*_{\bZ_d}\bZ_n$ and establish the key result in Lemma \ref{lemma:ker}.
This serves as a crucial component in the proof of our main theorem, Theorem \ref{main th}, concerning the first higher Kazhdan projection $p_1$ for $G$.

\begin{remark}
\label{remark: p1 only nonvaish}
It is known that the amalgamated product groups $G=\bZ_m*_{\bZ_d}\bZ_n$ are virtually free.
Therefore, by Lemma \ref{lemma: only p1 nonvainishing}, $p_1$ is the only non-vanishing higher Kazhdan projection in matrices over $\mathrm{C^*_{r}}(G)$, while the other $p_n$ vanish for $n\neq 1$.
\end{remark}

We start by expressing the first Laplacian $\Delta_1$ explicitly with Gruenberg's free resolution as introduced in Section \ref{subsection: Gruenberg resolution}.
	Consider the group $G=\bZ_m*_{\bZ_d}\bZ_n=\langle s,t|s^m=1,t^n=1,s^{\frac{m}{d}}=t^{\frac{n}{d}} \rangle$, where $d,m,n\in \mathbb{N}$ and $d|m$, $d|n$, $d\neq m$, $d\neq n$, $m \geq 2$ and $n\geq 3$. 
	Denote by
	\begin{equation*}
		r:=s^{\frac{m}{d}}=t^{\frac{n}{d}}.
	\end{equation*} 
	Let
	\begin{align*}
		p&=\frac{1+s+\cdots+s^{m-1}}{m},\\
		q&=\frac{1+t+\cdots+t^{n-1}}{n},\\
		h&=\frac{1+r+\cdots+r^{d-1}}{d}.
	\end{align*}
	It can be checked that $p,q,h$ are all projections on $\ell^2(G)$ and satisfy the following relations:
	\begin{align*}
		&ps=sp=p,\quad qt=tq=q,\\
		&ph=hp=p,\quad qh=hq=q.
	\end{align*}
    Furthermore, we consider the following factorizations of $p$ and $q$, respectively:
    \begin{equation}
    	\label{s1}
    	\begin{split}
    			p&=\frac{1}{m}(1+s^{\frac{m}{d}}+s^{\frac{2m}{d}}+\cdots+s^{m-\frac{m}{d}})(1+s+\cdots+s^{\frac{m}{d}-1})=\frac{d}{m}hf(s),\\
    			q&=\frac{1}{n}(1+t^{\frac{n}{d}}+t^{\frac{2n}{d}}+\cdots+t^{n-\frac{n}{d}})(1+t+\cdots+t^{\frac{n}{d}-1})=\frac{d}{n}hg(t),
    	\end{split}
    \end{equation}
    where we define elements $f(s)$ and $g(t)$ by
    \begin{equation*}
    	\begin{aligned}
    		f(s)&=1+s+\cdots+s^{\frac{m}{d}-1},\\
    		g(t)&=1+t+\cdots+t^{\frac{n}{d}-1}.
    	\end{aligned}
    \end{equation*}
	Then, by Theorem \ref{th1}, we obtain a free resolution of the trivial $\bZ G$-module $\bZ$ asscociated with the canonical presentation of $G$,
	\begin{equation}
    \label{free resolution1}
		\cdots \to (\bZ G)^{\oplus 3} \stackrel{\delta_1}{\longrightarrow}  (\bZ G)^{\oplus 2} \stackrel{\delta_0}{\longrightarrow}\bZ G \stackrel{\epsilon}{\longrightarrow} \bZ \longrightarrow 0.
	\end{equation}
    After applying the functor $\mathrm{Hom}_{\bZ G}(\cdot,\ell^2(G))$ to the above free resolution and identifying the terms with direct sums of $\ell^2(G)$, we obtain a cochain complex of Hilbert spaces:
    \begin{equation*}
    	\ell^2(G) \stackrel{d_0}{\longrightarrow} \ell^2(G)^{\oplus 2} \stackrel{d_1}{\longrightarrow} \ell^2(G)^{\oplus 3} \longrightarrow \cdots,
    \end{equation*}
    where $d_0$ and $d_1$ are the transposes of $\delta_0$ and $\delta_1$, respectively, given by:
    \begin{equation*}
    	d_0=\delta_0^{\mathsf{T}}=\left[\begin{matrix}
    		1-s\\
    		1-t
    	\end{matrix}\right],\quad
    	d_1=\delta_1^{\mathsf{T}}=\left[\begin{matrix}
    		\sum_{0\leq i \leq m-1}s^i & 0\\
    		0 & \sum_{0\leq j \leq n-1}t^j\\
    		-f(s) & g(t)
    	\end{matrix}\right],
    \end{equation*}
    followed from the Fox derivatives.
    We note that in fact we have
    \begin{equation*}
        \frac{\partial \left(s^{-\frac{m}{d}}t^{\frac{n}{d}}\right)}{\partial s}=s^{-\frac{m}{d}}f(s), \quad \frac{\partial \left(s^{-\frac{m}{d}}t^{\frac{n}{d}}\right)}{\partial t}=s^{-\frac{m}{d}}g(t).
    \end{equation*}
    However, for simplicity, we can omit the coefficient $s^{-\frac{m}{d}}$ in the associated entries of the homomorphisms $\delta_1$ and $d_1$, without affecting the exactness of the sequence (\ref{free resolution1}).
    This is because $s^{-\frac{m}{d}}$ is invertible in $\bZ G$ and commutes with every homomorphism of free $\bZ G$-modules, as it lies in the center of $\bZ G$.
    Therefore, the first Laplacian $\Delta_1$ on $\ell^2(G)^{\oplus 2}$ is
    \begin{align*}
    	\Delta_1&=d_0{d_0}^*+{d_1}^*d_1\\
     &=\left[\begin{matrix}
    		1-s\\
    		1-t
    	\end{matrix}\right]\left[\begin{matrix}
    	1-s^{-1} & 1-t^{-1}
        \end{matrix}\right] + \left[\begin{matrix}
        \sum_{0\leq i \leq m-1}s^i & 0 & -f(s)^* \\
        0 & \sum_{0\leq j \leq n-1}t^j &  g(t)^*
        \end{matrix}\right] \left[\begin{matrix}
        \sum_{0\leq i \leq m-1}s^i & 0\\
        0 & \sum_{0\leq j \leq n-1}t^j\\
        -f(s) & g(t)
        \end{matrix}\right]\\
    &=\left[\begin{matrix}
   	1-s\\
   	1-t
   \end{matrix}\right]\left[\begin{matrix}
   	1-s^{-1} & 1-t^{-1}
   \end{matrix}\right] + 
   \left[\begin{matrix}
   	\sum_{0\leq i \leq m-1}s^i & 0 \\
   	0 & \sum_{0\leq j \leq n-1}t^j 
   \end{matrix}\right] \left[\begin{matrix}
   	\sum_{0\leq i \leq m-1}s^i & 0\\
   	0 & \sum_{0\leq j \leq n-1}t^j
   \end{matrix}\right] \\
   &+
   \left[\begin{matrix}
   	-f(s)^*\\
   	g(t)^*   \end{matrix}\right]\left[\begin{matrix}
   	-f(s) & g(t)
   \end{matrix}\right]\\
      &=\left[\begin{matrix}
      	1-s & \ \\
      	\ & 1-t
      \end{matrix}\right]\left[\begin{matrix}
      1 & 1 \\
      1 & 1
     \end{matrix}\right]\left[\begin{matrix}
      1-s^{-1} & \ \\
      \ & 1-t^{-1}
     \end{matrix}\right] + \left[\begin{matrix}
     m^2p & \ \\
     \ & n^2q
     \end{matrix}\right] + \left[\begin{matrix}
     -f(s)^* & \ \\
     \ & g(t)^*
     \end{matrix}\right]\left[\begin{matrix}
      1 & 1 \\
      1 & 1
     \end{matrix}\right]\left[\begin{matrix}
      -f(s) & \ \\
          \ & g(t)
      \end{matrix}\right].
    \end{align*}

    Then we aim to calculate $\mathrm{ker}\Delta_1$ with coefficients in $\ell^2(G)$.
    In subsequent arguments, we shall adopt the following notation: 
    \begin{equation*}
    	C=\left[\begin{matrix}
    		p & \ \\
    		\ & q
    	\end{matrix}\right],
    \end{equation*}
    and we consider the following factorizations of $1-h$ for some $k$ and $l$, when $d\geq 2$:
    \begin{align*}
        1-h&=\frac{1}{d}\left[ (1-r)+\cdots +(1-r^{d-1})\right]=\frac{1}{d}\left[ (1-s^{\frac{m}{d}})+\cdots +(1-s^{\frac{m}{d}(d-1)})\right]\\&=k(1-s^{-1})=(1-s^{-1})k,\\
        1-h&=\frac{1}{d}\left[ (1-r)+\cdots +(1-r^{d-1})\right]=\frac{1}{d}\left[ (1-t^{\frac{n}{d}})+\cdots +(1-t^{\frac{n}{d}(d-1)})\right]\\
        &=l(1-t^{-1})=(1-t^{-1})l.
    \end{align*}
    Moreover, when $d=1$, it follows that $r=h=1$, and in this case, we set $k=l=0$.
    
    The proof of our main result on $\mathrm{ker}\Delta_1$ relies on the following technical lemma, which we present now.
    \begin{lemma}
    \label{lemma:restricted kernel trivial}
    The restricted operator $(f(s)k+g(t)l)|_{\mathrm{im}(1-h)}$ has trivial kernel, i.e., $$\mathrm{ker}\left(f(s)k+g(t)l\right)|_{\mathrm{im}(1-h)}=0.$$
  \end{lemma}
  \begin{proof}
   The case of $d=1$ is trivial since $h=1$ and we have set $k=l=0$.
   Thus, without loss of generality, we assume $d\geq 2$.
   Throughout this proof, we restrict all considered operators to the invariant subspace $\mathrm{im}(1-h)$ without further explanation, where the invariance is immediate. 
   Especially, the restricted operators $1-s, 1-s^{-1}, 1-t$, and $1-t^{-1}$ are all invertible. Then, we have
  	\begin{align}
  		&f(s)k+g(t)l \notag \\
  		&=f(s)(1-s^{-1})^{-1}+g(t)(1-t^{-1})^{-1} \notag\\
  		&=f(s)(1-s^{-1})^{-1}+g(t)(1-t)(1-t)^{-1}(1-t^{-1})^{-1} \label{s7}.
  	\end{align}
    Furthermore, observing that
  	\begin{equation*}
  		f(s)(1-s)=g(t)(1-t)=1-r,
  	\end{equation*}
    we may substitute $g(t)(1-t)$ with $f(s)(1-s)$ in the formula (\ref{s7}), which yields
    \begin{align}
    	&f(s)k+g(t)l \notag\\
        &=f(s)(1-s^{-1})^{-1}+f(s)(1-s)(1-t)^{-1}(1-t^{-1})^{-1}\notag\\
    	&=f(s)[(1-s^{-1})^{-1}+(1-s)(1-t)^{-1}(1-t^{-1})^{-1}]\notag\\
    	&=f(s)(1-s)[(1-s)^{-1}(1-s^{-1})^{-1}+(1-t)^{-1}(1-t^{-1})^{-1}]\label{s9}.
    \end{align}
    Next, observe that
    \begin{equation*}
        1-h=\frac{1}{d}\left[ (1-r)+\cdots +(1-r^{d-1})\right],\quad f(s)(1-s)=1-r,
    \end{equation*}
    which implies that $f(s)(1-s)$ is invertible on $\mathrm{im}(1-h)$. 
    Consequently, by equality (\ref{s9}), the conclusion of the lemma reduces to proving
  	\begin{equation*}
  		\mathrm{ker}\left((1-s)^{-1}(1-s^{-1})^{-1}+(1-t)^{-1}(1-t^{-1})^{-1}\right)|_{\mathrm{im}(1-h)}=0.
  	\end{equation*}
    Moreover, since we have the operator factorization:
  	\begin{align*}
  		&(1-s)^{-1}(1-s^{-1})^{-1}+(1-t)^{-1}(1-t^{-1})^{-1}\\
  		&=(1-s)^{-1}(1-s^{-1})^{-1}[(1-t^{-1})(1-t)+(1-s^{-1})(1-s)](1-t)^{-1}(1-t^{-1})^{-1},
  	\end{align*}
    it suffices to prove that 
  	\begin{equation*}
  		\mathrm{ker}\left((1-t^{-1})(1-t)+(1-s^{-1})(1-s)\right)|_{\ell^2(G)}=0,
  	\end{equation*}
    which follows directly from the argument below.
    Indeed, suppose $c\in \ell^2(G)$ satisfies
  	\begin{equation*}
  		(1-t^{-1})(1-t)c+(1-s^{-1})(1-s)c=0.
  	\end{equation*}
    Since the operators $(1-t^{-1})(1-t)$ and $(1-s^{-1})(1-s)$ are both positive definite, we deduce
    \begin{equation*}
	\left\{\begin{matrix}
		(1-t)c=0\\
		(1-s)c=0.
	\end{matrix} \right.
   \end{equation*}
    The above system of equations implies $c=0$ because the constant function in $\ell^2(G)$ is trivial, as $G$ is an infinite group by our assumptions about $d,m,n$.
    The conclusion follows. 
  \end{proof}

    Equipped with the above technical preparation, we now establish our main result on the calculation of $\mathrm{ker}\Delta_1$. 
    Although the proof follows from direct computation, it requires careful analysis of several technical details.
  \begin{lemma}
  \label{lemma:ker}
    Let $G=\bZ_m*_{\bZ_d}\bZ_n$ with $d|m$, $d|n$, $d\neq m$, $d\neq n$, $m \geq 2$ and $n\geq 3$.
    For $x\in \ell^2(G)^{\oplus2}$, $\Delta_1x=0$ if and only if there exist $z\in \ell^2(G)^{\oplus2}$ and $a\in \mathrm{im}h \cap \mathrm{im}(1-p) \cap \mathrm{im}(1-q)$ such that
    \begin{equation*}
    	x=(I-C)z,\quad \text{and} \quad \left[\begin{matrix}
    		1-s^{-1} & \ \\
    		\ & 1-t^{-1}
    	\end{matrix}\right]z=\left[\begin{matrix}
    	a \\
        -a
         \end{matrix}\right].
    \end{equation*}
  \end{lemma}
  \begin{proof}
  	Suppose $x\in \mathrm{ker}\Delta_1$. First, by applying $C$ to $\Delta_1x=0$, since $ps=p$ and $qt=q$, we have
  	\begin{equation}
  		\label{eq1}
  		\left[\begin{matrix}
  			m^2 & \ \\
  			\ & n^2
  		\end{matrix}\right]Cx+ \left[\begin{matrix}
  		-\frac{m}{d}p & \ \\
  		\ & \frac{n}{d}q
  	\end{matrix}\right]\left[\begin{matrix}
  	1 & 1 \\
  	1 & 1
  \end{matrix}\right]\left[\begin{matrix}
  -f(s) & \ \\
  \ & g(t)
  \end{matrix}\right]x=0.
  	\end{equation}
  	Then, by equality (\ref{s1}), we obtain
  	\begin{equation*}
  		pg(t)=\frac{d}{m}hf(s)g(t)=\frac{d}{m}h^2f(s)g(t)=\frac{n}{d}pq,
  	\end{equation*}
  	and similarly,
  	\begin{equation*}
  		qf(s)=\frac{m}{d}qp.
  	\end{equation*}
  	Thus, equation (\ref{eq1}) is equivalent to
  	\begin{equation*}
  		\left[\begin{matrix}
  			\left(\left(\frac{m}{d}\right)^2+m^2\right)p & -\frac{mn}{d^2}pq \\
  			-\frac{mn}{d^2}qp & \left(\left(\frac{n}{d}\right)^2+n^2\right)q
  		\end{matrix}\right]x=C\left[\begin{matrix}
  			\left(\left(\frac{m}{d}\right)^2+m^2\right) & -\frac{mn}{d^2}\\
  			-\frac{mn}{d^2} & \left(\left(\frac{n}{d}\right)^2+n^2\right)
  		\end{matrix}\right]Cx=0.
  	\end{equation*}
  	Furthermore, since there exists a unitary matrix $U_1$ such that 
  	\begin{equation*}
  		\left[\begin{matrix}
  			\left(\left(\frac{m}{d}\right)^2+m^2\right) & -\frac{mn}{d^2}\\
  			-\frac{mn}{d^2} & \left(\left(\frac{n}{d}\right)^2+n^2\right)
  		\end{matrix}\right]={U_1}^*\left[\begin{matrix}
  		\lambda_1 & \  \\
  		\  & \lambda_2
  	\end{matrix}\right]U_1,
  	\end{equation*}
  	where $\lambda_1, \lambda_2>0$ are eigenvalues of the real symmetric matrix on the left-hand side, equation (\ref{eq1}) is then equivalent to
  	\begin{equation*}
  		x\in \mathrm{ker}\left(\left[\begin{matrix}
  			\sqrt{\lambda_1} & \  \\
  			\  & \sqrt{\lambda_2}
  		\end{matrix}\right]U_1C\right)=\mathrm{ker}C=\mathrm{im}(I-C).
  	\end{equation*}
  	Therefore, equation (\ref{eq1}) holds for $x$ if and only if there exists $z\in \ell^2(G)^{\oplus2}$ such that
  	\begin{equation}
  		\label{ss3}
  		x=\mathrm{im}(I-C)z.
  	\end{equation}
  	Hence $\Delta_1x=0$ if and only if there exists $z$ such that
  	\begin{equation}
        \label{eq7}
        \begin{split}
  		\left[\begin{matrix}
  			1-s & \ \\
  			\ & 1-t
  		\end{matrix}\right]\left[\begin{matrix}
  			1 & 1 \\
  			1 & 1
  		\end{matrix}\right]&\left[\begin{matrix}
  			1-s^{-1} & \ \\
  			\ & 1-t^{-1}
  		\end{matrix}\right](I-C)z\\ &+ \left[\begin{matrix}
  			-f(s)^* & \ \\
  			\ & g(t)^*
  		\end{matrix}\right]\left[\begin{matrix}
  			1 & 1 \\
  			1 & 1
  		\end{matrix}\right]\left[\begin{matrix}
  			-f(s) & \ \\
  			\ & g(t)
  		\end{matrix}\right](I-C)z=0.
        \end{split}
  	\end{equation}
    Moreover, since equality (\ref{s1}) implies that
    \begin{equation*}
    	f(s)(1-p)=f(s)-\frac{m}{d}p=f(s)-hf(s)=(1-h)f(s),
    \end{equation*}
    and similarly,
  	\begin{equation*}
  		g(t)(1-q)=(1-h)g(t),
  	\end{equation*}
    together with equation (\ref{eq7}), we obtain that $\Delta_1x=0$ if and only if there exists $z$ satisfying
    \begin{equation}
    	\label{eq2}
    	\begin{split}
    			\left[\begin{matrix}
    				1-s & \ \\
    				\ & 1-t
    			\end{matrix}\right]\left[\begin{matrix}
    				1 & 1 \\
    				1 & 1
    			\end{matrix}\right]&\left[\begin{matrix}
    				1-s^{-1} & \ \\
    				\ & 1-t^{-1}
    			\end{matrix}\right]z\\ &+
    			\left[\begin{matrix}
    				-(1-h)f(s)^* & \ \\
    				\ & (1-h)g(t)^*
    			\end{matrix}\right]\left[\begin{matrix}
    				1 & 1 \\
    				1 & 1
    			\end{matrix}\right]\left[\begin{matrix}
    				-f(s) & \ \\
    				\ & g(t)
    			\end{matrix}\right]z=0.
    		\end{split}
    \end{equation}

    Next, we aim to solve the above equation (\ref{eq2}) for $z\in \ell^2(G)^{\oplus2}$.
    By applying $\left[\begin{matrix}
  		h & \ \\
  		\ & h
  	\end{matrix}\right]$ to equation (\ref{eq2}), we obtain
    \begin{equation}
    	\label{eq3}
    	\left[\begin{matrix}
    		h & \ \\
    		\ & h
    	\end{matrix}\right] \left[\begin{matrix}
    	1-s & \ \\
    	\ & 1-t
    \end{matrix}\right]\left[\begin{matrix}
    1 & 1 \\
    1 & 1
   \end{matrix}\right]\left[\begin{matrix}
   1-s^{-1} & \ \\
   \ & 1-t^{-1}
   \end{matrix}\right] \left[\begin{matrix}
    	h & \ \\
    	\ & h
    \end{matrix}\right]z=0.
    \end{equation}
   Then we consider the matrix decomposition:
   \begin{equation}
    	\label{s4}
    	\left[\begin{matrix}
    	1 & 1 \\
    	1 & 1
    \end{matrix}\right]=U\left[\begin{matrix}
    		2 & 0 \\
    		0 & 0 
    	\end{matrix}\right]U^*,
    \end{equation}
    where the unitary matrix $U$ is defined by 
  	\begin{equation*}
  		U=\frac{1}{\sqrt{2}}\left[\begin{matrix}
  			1 & 1 \\
  			1 & -1
  		\end{matrix}\right].
  	\end{equation*}
  	Thus, equation (\ref{eq3}) is equivalent to 
  	\begin{equation*}
  		\left[\begin{matrix}
  			\sqrt{2} & 0 \\
  			0 & 0 
  		\end{matrix}\right]U^*\left[\begin{matrix}
  		1-s^{-1} & \ \\
  		\ & 1-t^{-1}
  	\end{matrix}\right]\left[\begin{matrix}
  	h & \ \\
  	\ & h
    \end{matrix}\right]z=0,
  	\end{equation*}
  	which holds if and only if there exists $a\in \mathrm{im}h \cap \mathrm{im}(1-p) \cap \mathrm{im}(1-q)$ such that
  	\begin{equation}
  		\label{ss2}
  		\left[\begin{matrix}
  			1-s^{-1} & \ \\
  			\ & 1-t^{-1}
  		\end{matrix}\right]hz=\frac{1}{\sqrt{2}}\left[\begin{matrix}
  			a \\
  			-a
  		\end{matrix}\right].
  	\end{equation}
    Back to equation (\ref{eq2}), suppose there exists $a\in \mathrm{im}h \cap \mathrm{im}(1-p) \cap \mathrm{im}(1-q)$ such that equation (\ref{ss2}) holds, and consequently equation (\ref{eq3}) holds.
    Then under the space decomposition
    \begin{equation*}\ell^2(G)=h\ell^2(G)\oplus(1-h)\ell^2(G),
    \end{equation*} 
    we obtain that equation (\ref{eq2}) is equivalent to
  	\begin{align*}	&\left[\begin{matrix}
  			(1-h)(1-s) & \ \\
  			\ & (1-h)(1-t)
  		\end{matrix}\right]\left[\begin{matrix}
  			1 & 1 \\
  			1 & 1
  		\end{matrix}\right]\left[\begin{matrix}
  			(1-h)(1-s^{-1}) & \ \\
  			\ & (1-h)(1-t^{-1})
  		\end{matrix}\right]z\\ 
  	&+
  		\left[\begin{matrix}
  			-(1-h)f(s)^* & \ \\
  			\ & (1-h)g(t)^*
  		\end{matrix}\right]\left[\begin{matrix}
  			1 & 1 \\
  			1 & 1
  		\end{matrix}\right]\left[\begin{matrix}
  			-(1-h)f(s) & \ \\
  			\ & (1-h)g(t)
  		\end{matrix}\right]z=0.
  	\end{align*}
   Furthermore, clearly the above equation is equivalent to the following system of equations
  	\begin{numcases}{}
  		\left[\begin{matrix}
  			(1-h)(1-s) & \ \\
  			\ & (1-h)(1-t)
  		\end{matrix}\right]\left[\begin{matrix}
  			1 & 1 \\
  			1 & 1
  		\end{matrix}\right]\left[\begin{matrix}
  			(1-h)(1-s^{-1}) & \ \\
  			\ & (1-h)(1-t^{-1})
  		\end{matrix}\right]z=0 \label{eq4} \\
  		\left[\begin{matrix}
  			-(1-h)f(s)^* & \ \\
  			\ & (1-h)g(t)^*
  		\end{matrix}\right]\left[\begin{matrix}
  			1 & 1 \\
  			1 & 1
  		\end{matrix}\right]\left[\begin{matrix}
  			-(1-h)f(s) & \ \\
  			\ & (1-h)g(t)
  		\end{matrix}\right]z=0 \label{eq5}
  	.\end{numcases}
  On the one hand, by equality $(\ref{s4})$, clearly equation $(\ref{eq4})$ is equivalent to 
  \begin{equation*}
  	\left[\begin{matrix}
  		\sqrt{2} & 0 \\
  		0 & 0 
  	\end{matrix}\right]U^*\left[\begin{matrix}
  		1-s^{-1} & \ \\
  		\ & 1-t^{-1}
  	\end{matrix}\right]\left[\begin{matrix}
  		1-h & \ \\
  		\ & 1-h
  	\end{matrix}\right]z=0,
  \end{equation*}
  which holds if and only if there exists $b\in \mathrm{im}(1-h)\cap \mathrm{im}(1-p)\cap\mathrm{im}(1-q)=\mathrm{im}(1-h)$ such that 
  \begin{equation}
  	\label{s5}
  	\left[\begin{matrix}
  		1-s^{-1} & \ \\
  		\ & 1-t^{-1}
  	\end{matrix}\right](1-h)z=\frac{1}{\sqrt{2}}\left[\begin{matrix}
  		b \\
  		-b
  	\end{matrix}\right].
  \end{equation}
  On the other hand, similarly, equation $(\ref{eq5})$ is equivalent to 
  \begin{equation*}
  	\left[\begin{matrix}
  		\sqrt{2} & 0 \\
  		0 & 0 
  	\end{matrix}\right]U^*\left[\begin{matrix}
  	-(1-h)f(s) & \ \\
  	\ & (1-h)g(t)
  \end{matrix}\right]z=0.
  \end{equation*}
  We need to conduct a deeper analysis of the above equation.
  Recall that $1-h=k(1-s^{-1})=(1-s^{-1})k$ for some $k$, and $1-h=l(1-t^{-1})=(1-t^{-1})l$ for some $l$. Assuming the existence of $b\in \mathrm{im}(1-h)$ such that equation $(\ref{s5})$ holds, we obtain
  \begin{align*}
  	&\left[\begin{matrix}
  		\sqrt{2} & 0 \\
  		0 & 0 
  	\end{matrix}\right]U^*\left[\begin{matrix}
  		-(1-h)f(s) & \ \\
  		\ & (1-h)g(t)
  	\end{matrix}\right]z\\
    =&\left[\begin{matrix}
    	\sqrt{2} & 0 \\
    	0 & 0 
    \end{matrix}\right]U^*\left[\begin{matrix}
    	-(1-h)f(s) & \ \\
    	\ & (1-h)g(t)
    \end{matrix}\right]\left[\begin{matrix}
        k(1-s^{-1}) & \ \\
        \ & l(1-t^{-1})
    \end{matrix}\right](1-h)z\\
    =&\left[\begin{matrix}
    	1 & 0 \\
    	0 & 0 
    \end{matrix}\right]U^*\left[\begin{matrix}
    	-(1-h)f(s)k & \ \\
    	\ & (1-h)g(t)l
    \end{matrix}\right]\left[\begin{matrix}
        b \\
        -b
    \end{matrix}\right],
  \end{align*}
  which implies that at this time, equation $(\ref{eq5})$ is equivalent to 
  \begin{equation*}
  	\left[\begin{matrix}
  		1 & 0 \\
  		0 & 0 
  	\end{matrix}\right]U^*\left[\begin{matrix}
  		-(1-h)f(s)k & \ \\
  		\ & (1-h)g(t)l
  	\end{matrix}\right]\left[\begin{matrix}
  		b \\
  		-b
  	\end{matrix}\right]=0,
  \end{equation*}
  hence equivalent to
  \begin{equation}
  	\label{eq6}
   (f(s)k+g(t)l)(1-h)b=0.
  \end{equation}
  However, according to Lemma \ref{lemma:restricted kernel trivial}, the kernel of $f(s)k+g(t)l$ restricted to $\mathrm{im}(1-h)$ is trivial, thus equation $(\ref{eq6})$ holds if and only if 
  \begin{equation*}
  	(1-h)b=0.
  \end{equation*}
  Here notice that the above equation together with the condition $b\in \mathrm{im}(1-h)$ implies that $b$ vanishes.
  Finally, by equality $(\ref{s5})$ and the above discussion, we obtain that equations (\ref{eq4}) and (\ref{eq5}) both hold if and only if
  \begin{equation}
  	\label{s6}
  	 \left[\begin{matrix}
  	 	1-s^{-1} & \ \\
  	 	\ & 1-t^{-1}
  	 \end{matrix}\right](1-h)z=0.
  \end{equation}
  Therefore, equation (\ref{eq2}) holds if and only if there exists $a\in \mathrm{im}h \cap \mathrm{im}(1-p) \cap \mathrm{im}(1-q)$ such that equations (\ref{ss2}) and (\ref{s6}) both hold.

  Therefore, by equations (\ref{ss3}), (\ref{ss2}) and (\ref{s6}), the conclusion follows. Notice that the coefficient ${1}/{\sqrt{2}}$ in (\ref{ss2}) can be absorbed by rescaling $a\in \mathrm{im}h \cap \mathrm{im}(1-p) \cap \mathrm{im}(1-q) $.
  \end{proof}

\section{$K$-class of the higher Kazhdan projection for $\bZ_m*_{\bZ_d}\bZ_n$}
\label{section4}
 In this section, we establish the proof of our main result, Theorem \ref{main th}, which gives an explicit expression of the $K$-theory class $[p_1]\in \mathrm{K_0(C^*_r}(G))$ of the first higher Kazhdan projection for $G=\bZ_m*_{\bZ_d}\bZ_n$.
 The proof is based on several technical lemmas concerning orthogonal direct sum decompositions within multiple copies of $\ell^2(G)$, which we shall present first.
 These lemmas are straightforward extensions of the techniques from Pooya and Wang's work \cite{Pooya-Wang} to our setting.
  \begin{lemma}
  \label{lemma:H1}
   The first higher Kazhdan projection $p_1$, which is precisely the orthogonal projection onto $\mathrm{ker}\Delta_1$, shares the same $K$-theory class as the projection $P_{H_1}$ onto the subspace 
  	\begin{equation*}
  		H_1=\left\{\left[\begin{matrix}
  			a \\
  			-a
  		\end{matrix}\right] |\  a\in\mathrm{im}h \cap \mathrm{im}(1-p) \cap \mathrm{im}(1-q) \right\},
  	\end{equation*}
  	that is,
  	\begin{equation*}
  		[p_1]=[P_{H_1}]\in \mathrm{K_0(C^*_r}(G)).
  	\end{equation*}
  \end{lemma}
  \begin{proof}
  We consider the following factorizations of $1-p$ and $1-q$ for some $k_1$ and $l_1$, respectively:
  \begin{align}
  (1-p)&=(1-s^{-1})k_1=k_1(1-s^{-1}),\label{factorize1}\\
  (1-q)&=(1-t^{-1})l_1=l_1(1-t^{-1})\label{factorize2}.
  \end{align}
  Then by Lemma \ref{lemma:ker}, formulas (\ref{factorize1}) and (\ref{factorize2}), we infer that $\Delta_1x=0$ if and only if 
  	\begin{equation*}
  	x=(I-C)z=\left[\begin{matrix}
  			1-p & \\
  			& 1-q
  		\end{matrix}\right]z=\left[\begin{matrix}
  			k_1 & \\
  			& l_1
  		\end{matrix}\right]\left[\begin{matrix}
  			1-s^{-1} & \\
  			& 1-t^{-1}
  		\end{matrix}\right]z=\left[\begin{matrix}
  			{k_1}a \\
  			{-l_1}a
  		\end{matrix}\right],
  	\end{equation*}
  	for some $z\in \ell^2(G)^{\oplus 2}$ and $a\in\mathrm{im}h \cap \mathrm{im}(1-p) \cap \mathrm{im}(1-q)$. 
    Therefore, we obtain 
  	\begin{equation*}
  		\mathrm{ker}\Delta_1=\left\{\left[\begin{matrix}
  			{k_1}a \\
  			{-l_1}a
  		\end{matrix}\right]|\ a\in \mathrm{im}h \cap \mathrm{im}(1-p) \cap \mathrm{im}(1-q) \right\}.
  	\end{equation*}
    Furthermore, the factorization in formula (\ref{factorize1}) implies that $k_1$ and $1-s^{-1}$ are inverse to each other in $\mathrm{im}(1-p)$, while formula (\ref{factorize2}) implies that $l_1$ and $1-t^{-1}$ are inverse to each other in $\mathrm{im}(1-q)$.
    Construct the following bounded linear operator on $\ell^2(G)^{\oplus 2}$:
  	\begin{equation*}
  		V=\left[\begin{matrix}
  			1 & & & \\
  			& 1-s^{-1} & & \\
  			& & 1 & \\
  			& & & 1-t^{-1}
  		\end{matrix}\right]
  	\end{equation*}
    with respect to the decomposition $\mathrm{im}p\oplus\mathrm{im}(1-p)\oplus \mathrm{im}q\oplus\mathrm{im}(1-q)=\ell^2(G)^{\oplus 2}$.
    Clearly, $V$ is invertible and induces the following relations:
    \begin{equation*}
  		V\mathrm{ker}\Delta_1=H_1,\quad 	p_{1}=V^{-1}P_{H_1}V.
  	\end{equation*}
    Finally, the second relation above, together with the facts that $p_1$ belongs to the $C^*$-subalgebra $\mathrm{M}_{2}(\mathrm{C^*_{r}}(G))$ by Remark \ref{remark: p1 only nonvaish} and $V$ belongs to $\mathrm{M}_{2}(\mathrm{C^*_{r}}(G))$ by construction, implies that $P_{H_1}\in \mathrm{M}_{2}(\mathrm{C^*_{r}}(G))$. 
    The conclusion on $K$-theory classes then follows directly.
  \end{proof}
  
  Next, recalling the subspace $H_1 \subseteq \mathrm{im}h \oplus \mathrm{im}h$ defined in Lemma \ref{lemma:H1}, we construct four more subspaces of $\mathrm{im}h \oplus \mathrm{im}h \subseteq \ell^2(G)^{\oplus 2}$ as follows:
  \begin{align*}
  	H_2&=\left\{\left[\begin{matrix}
  		c \\
  		-c
  	\end{matrix}\right] |\  c\in\mathrm{im}p + \mathrm{im}q \right\},\quad
  	\tilde{H_2}=\left\{\left[\begin{matrix}
  		pw \\
  		qv
  	\end{matrix}\right] |\  v,w\in \ell^2(G) \right\},\\
  	&H_3=\left\{\left[\begin{matrix}
  		b \\
  		b
  	\end{matrix}\right] |\ b\in\mathrm{im}h \right\},\quad
  	\tilde{H_3}=\left\{\left[\begin{matrix}
  		(1-p)b \\
  		(1-q)b
  	\end{matrix}\right] |\ b\in\mathrm{im}h \right\}.
  \end{align*}
  We note that $H_2$ and $\tilde{H_2}$ are both subspaces of $\mathrm{im}h \oplus \mathrm{im}h$, since $\mathrm{im}p$ and $\mathrm{im}q$ are both contained in $\mathrm{im}h$.
  
  \begin{lemma}
  	\label{lemma1}
  	There are orthogonal direct sums of Hilbert spaces:
  	\begin{equation*}
  		\mathrm{im}h \oplus \mathrm{im}h = H_1\oplus H_2\oplus H_3,\quad \mathrm{im}h \oplus \mathrm{im}h = H_1\oplus \tilde{H_2}\oplus \tilde{H_3}.
  	\end{equation*}
  \end{lemma}
  \begin{proof}
  	First, we clearly have the orthogonal decomposition
  	\begin{equation*}
  		\ell^2(G)=\left(\mathrm{im}(1-p)\cap \mathrm{im}(1-q) \right) \oplus \left(\mathrm{im}p + \mathrm{im}q\right).
  	\end{equation*}
    Then based on the commutation relations
    $$(1-p)h=h(1-p),\quad (1-q)h=h(1-q), \quad ph=hp=p,\quad qh=hq=q,$$
    which can be checked directly, we obtain
    \begin{equation}
  		\label{s8}
  		\mathrm{im}h=h\ell^2(G)=\left(\mathrm{im}h \cap \mathrm{im}(1-p)\cap \mathrm{im}(1-q) \right)\oplus \left(\mathrm{im}p + \mathrm{im}q\right).
  	\end{equation}
  Hence, a direct verification shows the first conclusion of decomposition: 
  $$\mathrm{im}h \oplus \mathrm{im}h = H_1\oplus H_2\oplus H_3.$$

  Next, regarding the second decomposition of $\mathrm{im}h \oplus \mathrm{im}h$, we consider an arbitrary element of $H_1 + \tilde{H_2} + H_3$, which can be represented by
  	\begin{equation*}
  		\left[\begin{matrix}
  			a+pw+b \\
  			-a+qv+b
  		\end{matrix}\right],
  	\end{equation*}
  for some $a\in \mathrm{im}h \cap \mathrm{im}(1-p) \cap \mathrm{im}(1-q)$, $b\in \mathrm{im}h$ and $v,w\in \ell^2(G)$. 
  By the equality
  \begin{equation}
  \label{H1tH2H3}
  		\left[\begin{matrix}
  			a+pw+b \\
  			-a+qv+b
  		\end{matrix}\right]=\left[\begin{matrix}
  			a+pw+b \\
  			(a+pw+b)-(2a+pw-qv)
  		\end{matrix}\right],
  	\end{equation}
    we observe that when fixing $a$ and $w$, the term $a+pw+b$ covers the entire $\mathrm{im}h$ as $b$ ranges through $\mathrm{im}h$. 
    Meanwhile, as $a$, $w$, and $v$ vary over their respective domains, $2a+pw-qv$ also covers the whole $\mathrm{im}h$ by equality (\ref{s8}). 
    Therefore, we obtain elements in the form of (\ref{H1tH2H3}) covers
    \begin{equation*}
  		\left[\begin{matrix}
  			x \\
  			x-y
  		\end{matrix}\right],
  	\end{equation*} 
    for arbitrary $x,y\in \mathrm{im}h$, which implies
  	\begin{equation*}
  		\mathrm{im}h \oplus \mathrm{im}h = H_1 + \tilde{H_2} + H_3.
  	\end{equation*}
    Furthermore, since every element of $H_1 + \tilde{H_2} + H_3$ satisfies
  	\begin{equation*}
  		\left[\begin{matrix}
  			a \\
  			-a
  		\end{matrix}\right] + \left[\begin{matrix}
  			pw \\
  			qv
  		\end{matrix}\right] + \left[\begin{matrix}
  			b \\
  			b
  		\end{matrix}\right]=\left[\begin{matrix}
  			a \\
  			-a
  		\end{matrix}\right] + \left[\begin{matrix}
  			p(w+b) \\
  			q(v+b)
  		\end{matrix}\right] + \left[\begin{matrix}
  			(1-p)b \\
  			(1-q)b
  		\end{matrix}\right] \in H_1 + \tilde{H_2} + \tilde{H_3},
  	\end{equation*}
    for some $a\in \mathrm{im}h \cap \mathrm{im}(1-p) \cap \mathrm{im}(1-q)$, $b\in \mathrm{im}h$ and $v,w\in \ell^2(G)$, we infer that
  	\begin{equation*}
  		\mathrm{im}h \oplus \mathrm{im}h = H_1 + \tilde{H_2} + \tilde{H_3}.
  	\end{equation*}
  	Finally, it is straightforward to check that $H_1,\tilde{H_2}, \tilde{H_3}$ are mutually orthogonal, so we conclude that
  	\begin{equation*}
  		\mathrm{im}h \oplus \mathrm{im}h = H_1 \oplus \tilde{H_2} \oplus \tilde{H_3}.
  	\end{equation*}
    This finishes the proof.
  \end{proof}
  
  Denote by $P_{H_i}$ and $P_{\tilde{H_i}}$ the orthogonal projections from the Hilbert space $\ell^2(G)^{\oplus 2}$ onto subspaces $H_i$ and $\tilde{H_i}$, respectively, where $i\in \{2,3\}$.
  The following remark shows that the projections $P_{H_i}$ and $P_{\tilde{H_i}}$ all belong to the $C^*$-subalgebra $\mathrm{M}_{2}(\mathrm{C^*_{r}}(G))$, and hence they define classes in $\mathrm{K_0(C^*_r}(G))$. 

  \begin{remark}
  \label{remark}
      We first observe that 
      $P_{H_3}=T=\left[\begin{matrix}
  		\frac{h}{2} & \frac{h}{2} \\
  		\frac{h}{2} & \frac{h}{2}
  	\end{matrix}\right]\in \mathrm{M}_{2}(\mathrm{C^*_{r}}(G))$.
    Furthermore, since $T$ is unitarily equivalent to the projection $\left[\begin{matrix}
  		h & 0 \\
  		0 & 0
  	\end{matrix}\right]\in \mathrm{M}_{2}(\mathrm{C^*_{r}}(G))$,
    we infer that
  	\begin{equation}
        \label{H3}
  		[P_{H_3}]=[T]=[h]\in \mathrm{K_0(C^*_r}(G)).
  	\end{equation}
    In subsequent arguments, to simplify notations, we use $\mathrm{diag}(\cdot,\cdot)$ to denote a $2\times 2$ diagonal matrix.
    Then by Lemma \ref{lemma1}, we obtain
    $$
    P_{H_2}=\mathrm{diag}(h,h)-P_{H_1}-P_{H_3}\in \mathrm{M}_{2}(\mathrm{C^*_{r}}(G)).
    $$
    Next, by the fact that
    \begin{equation*}
        P_{\tilde{H_2}}= \mathrm{diag}(p,q)\in \mathrm{M}_{2}(\mathrm{C^*_{r}}(G)),
    \end{equation*}
    we also obtain
    $$
    P_{\tilde{H_3}}=\mathrm{diag}(h,h)-P_{H_1}-P_{\tilde{H_2}}\in \mathrm{M}_{2}(\mathrm{C^*_{r}}(G)),
    $$
    with the identity in $K$-theory:
    \begin{equation}
    \label{PtH2}
        [P_{\tilde{H_2}}]= [p]+[q]\in \mathrm{K_0(C^*_r}(G)).
    \end{equation}
  \end{remark}

  The next lemma establishes an identification between $K$-classes $[P_{H_3}]$ and $[P_{\tilde{H_3}}]\in \mathrm{K_0(C^*_r}(G))$.
  \begin{lemma}
  	\label{lemma2}
  	\begin{equation*}
  		[P_{H_3}]=[P_{\tilde{H_3}}] \in \mathrm{K_0(C^*_r}(G)).
  	\end{equation*}
  \end{lemma}
  \begin{proof}
  The proof is similar to that of the case $G=\bZ_m*\bZ_n$, which was verified in \cite[Lemma 4.8]{Pooya-Wang}.
  The main idea is to construct an invertible operator $U\in \mathrm{M}_{2}(\mathrm{C^*_{r}}(G))$ to implement the conjugacy between $P_{H_3}$ and $P_{\tilde{H_3}}$. Moreover, $U$ is based on two invertible operators $u:\tilde{H_2}\to H_2$ and $v:H_3\to \tilde{H_3}$, which remain invertible in our broader case of $G=\bZ_m*_{\bZ_d}\bZ_n$.
  \end{proof}
  With all the ingredients, we are now ready to present the proof for Theorem \ref{main th}.
  \begin{proof}[Proof of Theorem \ref{main th}]
    As Remark \ref{remark} shows that $P_{\tilde{H_2}}$ and $P_{\tilde{H_3}}$ are both projections over $\mathrm{C^*_{r}}(G)$, by Lemma \ref{lemma1}, we infer that
     \begin{equation*}
  		[P_{H_1}]=2[h]-[P_{\tilde{H_2}}]-[P_{\tilde{H_3}}]\in \mathrm{K_0(C^*_r}(G)).
  	\end{equation*}
    The above equality, together with formulas (\ref{H3}) and (\ref{PtH2}), and Lemma \ref{lemma2}, implies that
  	\begin{equation*}
  		[P_{H_1}]=2[h]-[p]-[q]-[h]=[h]-[p]-[q]\in \mathrm{K_0(C^*_r}(G)).
  	\end{equation*}
    Therefore, by Lemma \ref{lemma:H1}, we deduce that
    \begin{equation*}
  		[p_1]=[P_{H_1}]=[h]-[p]-[q]\in \mathrm{K_0(C^*_r}(G)),
  	\end{equation*}
    which completes the proof.
  \end{proof}
   Applying delocalized traces $\tau_{\langle g\rangle}$ to both sides of the formula for $[p_1]$ in Theorem \ref{main th}, we derive the first delocalized $\ell^2$-Betti numbers for $G=\bZ_m*_{\bZ_d}\bZ_n$ as presented in the following corollary. Moreover, the delocalized $\ell^2$-Betti numbers for $G$ vanish in degrees other than 1 due to Remark \ref{remark: p1 only nonvaish}.
  \begin{corollary}
The delocalized $\ell^2$-Betti numbers for $G=\bZ_m*_{\bZ_d}\bZ_n$, with $d|m$, $d|n$, $d\neq m$, $d\neq n$, $m \geq 2$ and $n\geq 3$, are:
        \begin{equation*}
            \beta^{(2)}_{1, \langle{g}\rangle}(G) =\frac{|\langle g \rangle \cap \bZ_d|}{d}-\frac{|\langle g \rangle \cap \bZ_m|}{m}-\frac{|\langle g \rangle \cap \bZ_n|}{n},\quad g\in G,
        \end{equation*}
		and $\beta ^{(2)}_{k, \langle{g}\rangle}(G) = 0$ for $k \neq 1,\ g\in G$.
        Especially, $\beta^{(2)}_{1, \langle{g}\rangle}(G)=\frac{1}{d}-\frac{1}{m}-\frac{1}{n}$.
    \end{corollary}
   In particular, when $m=4$, $n=6$ and $d=2$, we obtain the results associated with $G=\mathrm{SL}(2,\bZ)=\bZ_4*_{\bZ_2}\bZ_6$, as presented in Corollary \ref{corollary:SL2} in Section \ref{section1}.
  
  

  	
  	
	
	
	\bibliographystyle{alpha}
	\bibliography{mybib}

	\vspace{2em}
	\begin{minipage}[t]{0.45\linewidth}  
        
        \small 
		Baiying Ren \\ Research Center of Operator Algebras \\
		East China Normal University\\
		Shanghai 200241, China \\
		{\footnotesize 52275500020@stu.ecnu.edu.cn}
	\end{minipage}

\end{document}